\documentclass[12pt]{amsart}
\usepackage{amsmath}

\DeclareMathOperator{\Pic}{Pic}
\DeclareMathOperator{\Tor}{Tor}

\DeclareMathOperator{\Sym}{Sym}

\DeclareMathOperator{\B}{{B}}
\DeclareMathOperator{\N}{\mathbb{N}}
\def\A{{A}}

\def\E{{E}}
\def\F{\mathcal{F}}

\def\L{{L}}
\def\M{{M}}
\def\N{\mathbb{N}}

\def\O{\mathcal{O}}
\def\P{\mathbb{P}}

\def\R{\mathbb{R}}

\def\Z{\mathbb{Z}}

\def\u{{\bf u}}
\def\v{\bf{v}}
\def\w{\bf{w}}
\def\<{\langle}
\def\>{\rangle}
\def\({\left(}
\def\){\right)}

\theoremstyle{plain}
\newtheorem{theorem}{Theorem}

\newtheorem{corollary}[theorem]{Corollary}
\newtheorem{lemma}[theorem]{Lemma}

\theoremstyle{definition}
\newtheorem{definition}[theorem]{Definition}

\newtheorem{example}[theorem]{Example}
\newtheorem{remark}[theorem]{Remark}

\begin{document}
\bibliographystyle{alpha}

\date{\today}
\title{Multigraded regularity and the Koszul property}
\author[M.~Hering]{Milena Hering}
\address{Institute of Mathematics and its Applications,
University of Minnesota, 400 Lind Hall, Minneapolis, MN 55455, USA}
\email{{\tt hering@ima.umn.edu}}

\begin{abstract}
  We give a criterion for the section ring of an ample line bundle 
  to be Koszul in terms of multigraded regularity. We discuss an application 
  to polytopal semigroup rings.  
\end{abstract}
\maketitle

\section{Introduction}

Let $\A$ be an ample line bundle 
on a projective variety $X$ over a field $k$, and let 
$R\(\A\)=\bigoplus_{\ell\geq 0}H^0\(X,\A^{\ell}\)$ be the section ring 
associated to $\A$. 
Recall that a graded $k$ algebra $R=k\oplus R_1\oplus R_2 \oplus \cdots$ 
is called \emph{Koszul} (or \emph{wonderful})
if $k$ admits a linear free resolution over $R$. 
It is well known that a Koszul algebra is generated in degree 1, and 
that the ideal of relations between its generators is generated by quadrics.

The purpose of this note is to give criteria for the section ring of 
an ample line bundle to be Koszul in terms of the regularity of the 
line bundle.  The following theorem illustrates the flavour of our 
main result, Theorem \ref{thm:multigraded}. 

\begin{theorem}\label{thm:easy}
  Let $\A$ be an ample line bundle on a projective variety $X$ over an 
  infinite field $k$. Assume that 
  $H^i(X,A^{m-i}) = 0$ for $i>0$. Then the section ring 
  $R(\A^m) = \bigoplus _{\ell\geq 0} H^0(X,\A^{\ell m})$ is Koszul. 
\end{theorem}

It is well known that section rings of high enough powers of ample line bundles 
are Koszul (see \cite{Backelin86}). 
Moreover, Eisenbud and Reeves \cite{EisenbudReevesTotaro} 
give criteria for 
Veronese subalgebras of 
graded $k$-algebras to be Koszul in terms of the 
algebraic Castelnuovo-Mumford regularity. We illustrate 
the relationship between our theorem and these criteria in 
the end of Section \ref{sec:proof}. 
  
Sufficient criteria for powers of ample line bundles 
to have Koszul section ring  are known for 
curves \cite{VishikFinkelberg93,Butler94, Polishchuk95,PareschiPurnaprajna97,
ConcaRossiValla}, homogeneous spaces 
\cite{InamdarMehta94,Bezrukavnikov95,Ravi95}, 
elliptic ruled surfaces \cite{GallegoPurnaprajna96a}, abelian varieties 
\cite{Kempf89}, and toric varieties
\cite{BrunsGubeladzeTrung97}. 
More generally,  
there are criteria for certain
 adjoint line bundles on smooth projective varieties 
to have Koszul section ring, see 
 \cite{Pareschi93}.  

The Koszul property has also been studied
for 
points in projective spaces, see \cite{ConcaTrungValla,Kempf92,Polishchuk06}, 
and toric varieties admitting additional combinatorial 
structure, see \cite{Sturmfels96, 
PeevaReinerSturmfels, HerzogHibiRestuccia, OhsugiHibi99K}.
The Koszul property appears naturally in many areas 
of mathematics; see for example \cite{PolishchukPositselski} 
for an introduction to Koszul algebras from different perspectives.

We will prove a more general version of Theorem \ref{thm:easy} 
in terms of multigraded regularity 
(compare \cite{MaclaganSmith04} and \cite{HeringSchenckSmith}), see Theorem 
\ref{thm:multigraded}. It generalizes a result for line bundles 
on surfaces by Gallego and Purnaprajna 
\cite[Theorem 5.4]{GallegoPurnaprajna96} that they use to give exact 
criteria for line bundles on elliptic ruled surfaces to have a Koszul 
section ring. 
The proof is based on a vanishing theorem due to 
Lazarsfeld and uses methods very similar to those 
in \cite{GallegoPurnaprajna96} and \cite{HeringSchenckSmith}.

As an application, we show how the criteria for polytopal semigroup 
rings to be Koszul due to Bruns, Gubeladze and 
Trung \cite[Theorem 1.3.3.]{BrunsGubeladzeTrung97} can be 
improved if multiples of the polytope do not contain 
interior lattice points, see Section \ref{sec:polytopes}.  
Theorem \ref{thm:multigraded} is part of my thesis \cite{Hering05}.

\subsection*{Acknowledgements}
  I would like to thank W.~Fulton, R.~Lazarsfeld, E.~Miller,
 M.~ Musta{\c t}a, 
 H.~Schenck, and G.~Smith. Furthermore I would like to thank the 
Institute of Mathematics and its Applications and the Max Planck 
Institute f{\"u}r Mathematik, Bonn for their hospitality.

\section{Multigraded regularity and proof of theorem
}\label{sec:proof}

Let $X$ be a projective variety over a field $k$.
We will assume for the remainder of the paper that $k$ is infinite. 
Let $\B_1, \ldots, \B_r$ be globally generated line bundles 
on $X$.
For $\u\in \Z^r$, we let 
$\B^{\u}:= \B_1^{u_1}\otimes \cdots
\otimes \B_r^{u_r}$ and $|\u|=u_1 +\cdots +u_r$. 
Let  $\mathcal{B} = \{ \B^{\u} \mid \u\in \N^r\} \subset \Pic(X)$ be the 
submonoid of $\Pic(X)$ 
generated by $\B_1, \ldots, \B_r$. 

\begin{definition}\label{def:regularity} 
Let $\L$ be a line bundle on $X$. 
A sheaf $\F$ is called
$\L$-regular with respect to 
$\B_1, \ldots, \B_r$ if 
\begin{equation*}
  H^i\left(X, \F\otimes \L\otimes\B^{-\u}\right) =0
\end{equation*}
for all $i>0$ and for all $\u\in \N^r$ with $|\u| 
= i$.
\end{definition}

Obeserve that for $r=1$, this is the usual definition 
for Castelnuovo-Mumford regularity, compare 
\cite[1.8.4.]{Lazarsfeld04a}. 

Now we are ready to state the main theorem. 

\begin{theorem}\label{thm:multigraded}
Let $\B_1, \ldots , \B_r$ be a set of globally generated 
line bundles on $X$ generating a semigroup $\mathcal{B}$.
Let $\A\in \mathcal{B}$ be an ample line bundle such that 
$\A\otimes \B_i^{-1} \in {{\mathcal{B}}}$ for all $i=1,\ldots,r$.
If $\A$ is $\O_X$-regular with respect to $\B_1, \ldots, \B_r$,   
then 
the section ring  
 $R(\A)=\bigoplus_{\ell\geq 0} H^0(X,\A^{\ell})$ is Koszul. 
\end{theorem}

In the proof we will use the following generalization of Mumford's 
theorem to the 
multigraded case. For ease of notation, we fix $B_1, \ldots, B_r$, and 
say a sheaf $\F$ is $L$-regular if it is $L$-regular with 
respect to $B_1, \ldots, B_r$.  

\begin{theorem}[\cite{HeringSchenckSmith}, Theorem 2.1]\label{thm:Mumford}
  Let $\mathcal{F}$ be $\L$-regular.  
  Then for all
  ${\u} \in \mathbb{N}^{r}$,
  \begin{enumerate}
  \item $\mathcal{F}$ is $\(\L \otimes B^{{\u}}\)$-regular;
  \item the natural map 
    \[
    H^{0} \(X, \mathcal{F} \otimes \L \otimes B^{{\u}}\) \otimes
    H^{0}\(X, B^{{\v}}\) \to H^{0}\(X, \mathcal{F} \otimes \L \otimes
    B^{{\u}+{\v}}\)
    \] 
    is surjective for all ${\v} \in \mathbb{N}^{r}$;
  \item 
    $\mathcal{F} \otimes \L \otimes B^{{\u}}$ is generated by its
    global sections, provided there exists ${\w} \in
    \mathbb{N}^{r}$ such that $B^{{\w}}$ is ample.
  \end{enumerate}
\end{theorem}

Observe that the proof of Theorem  
\ref{thm:Mumford} in \cite{HeringSchenckSmith}
only requires $k$ to be infinite. 

We will also make ample use of the following lemma. 
\begin{lemma}[\cite{HeringSchenckSmith}, Lemma 2.2]\label{lem:ses}
  Let $0 \to \mathcal{F}' \to \mathcal{F} \to \mathcal{F}'' \to 0$
  be a short exact sequence of coherent $\mathcal{O}_{X}$-modules.  If
  $\mathcal{F}$ is $\L$-regular, $\mathcal{F}''$ is $\(\L \otimes
  B_j^{-1}\)$-regular for all $1 \leq j \leq r$ and $H^{0}\(X,
  \mathcal{F} \otimes \L \otimes B_j^{-1}\) \to H^{0}\(X,
  \mathcal{F}'' \otimes \L \otimes B_j^{-1}\)$ is surjective for
  all $1 \leq j \leq r$, then $\mathcal{F}'$ is also $\L$-regular.
\end{lemma}

The main tool of the proof is a vanishing theorem for a family 
of vector bundles associated to an ample and globally generated 
line bundle $\A$ as follows. To a globally generated vector 
bundle $\E$ is associated a vector bundle 
$\M_{\E}$, the kernel of the evaluation map 
\begin{equation}\label{eq:ME}
    0\to \M_{\E}\to H^0\left(X,\E\right)\otimes \O_X \to \E\to 0. 
  \end{equation}

For $h\in \N$, we define vector bundles $\M^{\left(h\right)}$ inductively, by 
letting $\M^{\left(0\right)}= \A$ and $\M^{\left(h\right)} = \M_{\M^{\left(h-1\right)}}\otimes \A$, 
provided $\M^{\left(h-1\right)}$ is globally generated.

\begin{lemma}[Lazarsfeld, see {\cite[Lemma 1]{Pareschi93}}]\label{lem:Koszul}
  Let $X$ be a projective variety, 
  let $\A$ be an  ample line bundle on $X$, and let 
$R(\A)=\bigoplus_{\ell\geq 0} H^0\(X,\A^{\ell}\)$ 
 be the 
  section ring associated to $\A$. Assume that the 
  vector bundles $\M^{\(h\)}$ are globally generated for all $h\geq 0$. If 
  $H^1\left(X,\M^{\(h\)}\otimes \A^{\ell}\right)=0$ for all $\ell \geq 0$ 
  then $R(\A)$
  is Koszul. Moreover, if $H^1\(X,\A^{\ell}\)=0$ for all 
  $\ell \geq 1$, the converse also holds. 
\end{lemma}
 
Observe that the proof of this lemma is valid for projective 
varieties over any field. 

\begin{proof}[Proof of Theorem \ref{thm:multigraded}]

  We will use induction on $h$ to show 
 that $\M^{\left(h\right)}$ is $\O_X$-regular. In 
  particular, by Theorem \ref{thm:Mumford}, (3),
  $\M^{\left(h\right)}$ is globally generated, and  $\M^{h+1}$ 
  is defined.

  Tensoring \eqref{eq:ME} for $E=\M^{\left(h-1\right)}$ with $\A$, we obtain the 
  following short exact sequence 
  \begin{equation*}
    0\to \M^{\left(h\right)} \to H^0\left(X,\M^{\left(h-1\right)}\right)\otimes \A 
    \to \M^{\left(h-1\right)}\otimes \A
    \to 0.
  \end{equation*}
  Then 
  $H^0\left(X,\M^{\left(h-1\right)}\right)\otimes \A$ is $\O_X$ regular.
  Since $\A\otimes B^{-e_j} \in \mathcal{B}$, it follows that $\A\otimes B^{-e_j}\cong
  B^{u'}$ for $u'\in \N^r$. 
  By the induction hypothesis and Theorem \ref{thm:Mumford}, (1) 
  $\M^{\left(h-1\right)}$ is $\A \otimes B^{-e_j}$-regular for all $j$, and 
  so $\M^{\left(h-1\right)}\otimes \A$ is $B^{-e_j}$-regular for all $j$.  
  Similarly, by Theorem \ref{thm:Mumford}, (2), the natural map
  $H^0\left(X, \M^{\left(h-1\right)}\right)\otimes H^0\left(X,\A\otimes  B^{-e_j}\right) 
  \to 
  H^0\left(X,\M^{\left(h-1\right)}\otimes \A\otimes B^{-e_j}\right)$ is surjective for all 
  $1\leq j\leq r$. Applying Lemma \ref{lem:ses}, we see that $\M^{\left(h\right)}$
  is $\O_X$-regular. 

  Theorem \ref{thm:Mumford}, (1) implies
  $\M^{h}$ is also $\B^u$-regular for all
  $u\in \N^r$.
  Hence $H^1\left(X,\M^{\left(h\right)}\otimes \A^{\ell}\right) = 0$ for all 
  $h\geq 0$ and $\ell \geq 0$, and 
  Lemma \ref{lem:Koszul} implies that  $R\left(\A\right)$ is 
  Koszul. 
\end{proof}

Theorem \ref{thm:easy} is the special case when 
$r=1$. 

\begin{example} 
The fact that the section ring of high enough powers of ample line bundles 
is Koszul follows easily from this result:
By Serre vanishing, $L^{d}$ is 
$\O_X$-regular with respect to $L$ for $d$ large enough,
hence the associated section ring is Koszul.  
\end{example}

\begin{remark}
Let $R\cong k[x_0, \ldots, x_N]/I$, where $I\subset k[x_0, \ldots, 
x_N]$ is  a homogeneous ideal. 
If $I$ admits a quadratic Gr{\"o}bner basis with respect to 
some monomial ordering, then $R$ is Koszul. 
However, 
a Koszul algebra need not admit a presentation whose ideal 
admits a quadratic Gr{\"o}bner basis, see \cite{EisenbudReevesTotaro}.
\end{remark}
  
\begin{remark} 
It is well known that if the section ring of a line bundle 
$L$ is Koszul, then $L$ satisfies Green's property 
$N_1$
(see \cite[1.8.C]{Lazarsfeld04a} for an introduction to property $N_p$).
On the other hand, 
Sturmfels \cite[Theorem 3.1]{Sturmfels00}
exhibited an example of a smooth projectively normal curve
whose coordinate ring is presented by quadrics but is  not Koszul.
However, in many cases, 
criteria  for line bundles to satisfy $N_p$ 
imply 
that their section ring is Koszul  
and even that its ideal admits 
a quadratic Gr{\"o}bner basis when $p\geq 1$. 
For example,
the conditions for Theorem  \ref{thm:multigraded} agree with those of \cite{HeringSchenckSmith} 
for a line bundle to satisfy $N_1$. 
\end{remark}

Eisenbud, Reeves and 
Totaro \cite{EisenbudReevesTotaro} 
give criteria for 
Veronese subrings of finitely generated graded $k$-algebras 
to be Koszul in terms of algebraic regularity. 
Translating their result into the language of ample line bundles, 
we obtain a better 
bound than Theorem \ref{thm:easy}  for normally generated line bundles. 

\begin{definition}
 An ample line bundle is called \emph{normally generated}, if 
 the natural map 
 \[
 \underbrace{H^0(X,L)\otimes \cdots \otimes H^0(X,L)}_m \to 
 H^0(X,L^m) \]
 is surjective for all $m$.  
\end{definition}

\begin{corollary}\label{cor:ERT}
Let $A$ be a normally generated line bundle. Suppose 
$A^m$ is $\O_X$-regular with respect to itself. Then if 
$d\geq \frac{m}{2}$, the ideal of the section ring associated to $A^d$ 
admits a quadratic Gr{\"o}bner basis; 
in particular, the section ring is Koszul. 
\end{corollary}

To see how this theorem follows from the criteria in \cite{EisenbudReevesTotaro}, we first review the notion of algebraic regularity. 

\begin{definition}\label{def:algregularity}
Let $S=k[x_0, \ldots, x_N]$ be a  polynomial 
ring over $k$. 
A finitely generated graded $S$-module $M$ is 
$m$-regular 
if $\Tor^S_i(M,k)_j=0$ for $j>i+m$ and $i\geq 0$.  
\end{definition}

Let $I\subset S$ be a homogeneous ideal, and let  $R=S/I$.   
We denote with $R^{(d)} = \bigoplus _{m\in \N} R_{md}$ the 
$d$'th Veronese subalgebra of $R$.
Keeping in mind that $S/I$ is $m$-regular if and 
only if $I$ is $(m+1)$-regular, 
the following theorem is an easy  
consequence of the results proved in 
Eisenbud, Reeves and Totaro \cite{EisenbudReevesTotaro}.

\begin{theorem}[\cite{EisenbudReevesTotaro}]\label{thm:ERT}
If  $R$ is $(m-1)$-regular and 
$d\geq \frac{m}{2}$, then the ideal of $R^{(d)}$ 
admits a  quadratic Gr{\"o}bner basis.  
\end{theorem}

\begin{proof}[Proof of Corollary \ref{cor:ERT}]
Since $A$ is normally generated, it is very ample. 
Moreover, the section ring $R$ 
associated to $A$ is generated in degree $1$, and it agrees with the 
homogeneous coordinate ring of the embedding 
$\iota\colon X\hookrightarrow \P:= \P\left(H^0\left(X,A\right)\right)$ 
induced by $A$. 
In particular, $R$ is 
of the form $S/I$ for $S=\Sym^{\bullet}H^0(X,A)$ and $I$ a homogeneous
ideal in $S$. 
Now 
$A^m$ is $\O_X$-regular  with respect to itself 
if and only if $\iota_*A$ is $\O_{\P}(m-1)$-regular 
 with respect to 
$\O_{\P}(1)$ 
if and only if  
$R(\iota_*A)=R(A)$ is $(m-1)$-regular as a $S$-module, 
 (see 
for example \cite[Exercise 20.20.]{Eisenbud95} or 
\cite[1.8.26.]{Lazarsfeld04a}). Since the  
section ring associated to $A^d$ agrees with $R^{(d)}$, the corollary 
follows from  Theorem \ref{thm:ERT}. 
\end{proof}
\section{Polytopal semigroup rings}\label{sec:polytopes} The question which powers of ample line bundles on toric 
varieties have Koszul section ring was studied 
by Bruns, Gubeladze and Trung in \cite{BrunsGubeladzeTrung97}. 
They prove that for 
an ample line bundle $A$  on a toric variety $X$ of dimension $n$,
the section ring $R(A^n)$ is Koszul. 
Observe that this also follows easily 
from Theorem \ref{thm:easy}. In fact, 
since the higher cohomology of an ample line bundle 
on a toric variety vanishes, $A^n$ is $\O_X$-regular. 

A more careful study of the regularity of a line bundle 
on a toric variety shows that if $r$ is the number of integer 
roots of the Hilbert polynomial of $A$, then  
$A^{n-r}$ is $\O_X$-regular (see \cite[Lemma 4.1]{HeringSchenckSmith}),
and we obtain the following Corollary. 

\begin{corollary}
Let $A$ be an ample line bunlde on a toric variety $X$, and let 
$r$ be the number of integer roots of the Hilbert polynomial 
of $A$. 
Then $R(A^{n-r})$ is Koszul. 
\end{corollary}

In terms of lattice polytopes and polytopal semigroup rings 
this can be rephrased as follows. Let $M\cong \Z^n$ be a lattice, 
and $P\subset M\otimes \R:=M_{\R}$ be a lattice polytope.  
$P$ determines a semigroup $S_P \subset M \times \Z$, the 
semigroup generated by $\{(p,1)\in M\times \Z \mid p\in P\cap M\}$. 
Let $k[S_P]$ be the semigroup algebra associated to $S_P$. 

\begin{corollary}
 Let $P$ be a lattice polytope of 
dimension $n$, and let $r$ be the largest positive integer such that 
$rP$ does not contain any interior lattice points. Then the  
polytopal semigroup ring $k[S_{(n-r)P}]$ is Koszul. 
\end{corollary}

This follows from the fact that $r$ is
the number of integer roots of the Hilbert polynomial of the ample 
line bundle associated to $P$ (see for example \cite[Section 4]{HeringSchenckSmith}). 

\bibliography{/Users/mhering/Documents/mybib}
\end{document}